\documentclass[a4paper]{article} 

\usepackage[utf8]{inputenc}
\usepackage{amsthm}
\usepackage{amsmath,amssymb,url}
\usepackage{enumitem}
\usepackage[pdfencoding=auto,hidelinks]{hyperref}
\usepackage{mathtools}
\usepackage{authblk}
\usepackage[linesnumbered]{algorithm2e}

\usepackage[left=3cm,right=3cm,top=
2.5cm,bottom=2cm,includeheadfoot]{geometry}

\newcommand{\AN}{\overline{\mathbb{Q}}}
\newcommand{\CN}{\mathbb{C}}
\newcommand{\QN}{\mathbb{Q}}
\newcommand{\KN}{\mathbb{K}}
\newcommand{\LN}{\mathbb{L}}
\newcommand{\NN}{\mathbb{N}}
\newcommand{\ZN}{\mathbb{Z}}

\newcommand{\st}{\mid}

\newcommand{\forAlln}{\quad \text{ for all } n \in \NN}
\newcommand{\CFin}{\ensuremath{\mathcal{R}_C}}
\newcommand{\fractions}[1]{Q(#1)} 
\newcommand{\shift}{\sigma}
\newcommand{\op}[1]{\mathcal{\MakeUppercase{#1}}} 

\newcommand{\diffalg}[2]{#1_\shift[#2]} 

\newcommand{\torsion}[1]{T(#1)} 
\def\i{\mathrm{i}} 
\newcommand{\norm}[1]{\left\lVert#1\right\rVert_{2}}
\newcommand{\norminf}[1]{\left\lVert#1\right\rVert_\infty}
\newcommand{\abs}[1]{\left\lvert #1 \right\rvert}

\DeclareMathOperator{\diag}{diag} 
\DeclareMathOperator{\ord}{ord}

\DeclareMathOperator{\lcm}{lcm}

\theoremstyle{definition}
\newtheorem{dfn}{Definition}
\newtheorem{thm}[dfn]{Theorem}
\newtheorem{lem}[dfn]{Lemma}
\newtheorem{ex}[dfn]{Example}

\RestyleAlgo{ruled}

\makeatletter
\renewcommand{\SetKwInOut}[2]{%
  \sbox\algocf@inoutbox{\KwSty{#2}\algocf@typo:}%
  \expandafter\ifx\csname InOutSizeDefined\endcsname\relax
    \newcommand\InOutSizeDefined{}\setlength{\inoutsize}{\wd\algocf@inoutbox}%
    \sbox\algocf@inoutbox{\parbox[t]{\inoutsize}{\KwSty{#2}\algocf@typo:\hfill}~}\setlength{\inoutindent}{\wd\algocf@inoutbox}%
  \else
    \ifdim\wd\algocf@inoutbox>\inoutsize%
    \setlength{\inoutsize}{\wd\algocf@inoutbox}%
    \sbox\algocf@inoutbox{\parbox[t]{\inoutsize}{\KwSty{#2}\algocf@typo:\hfill}~}\setlength{\inoutindent}{\wd\algocf@inoutbox}%
    \fi%
  \fi
  \algocf@newcommand{#1}[1]{%
    \ifthenelse{\boolean{algocf@inoutnumbered}}{\relax}{\everypar={\relax}}%
    {\let\\\algocf@newinout\hangindent=\inoutindent\hangafter=1\parbox[t]{\inoutsize}{\KwSty{#2}\algocf@typo:\hfill}~##1\par}%
    \algocf@linesnumbered
  }}%
\makeatother

\title{Order bounds for $C^2$-finite sequences\thanks{The research was funded by the Austrian Science Fund (FWF) under the grant W1214-N15, project DK15 and under the grants P33530, P31571, and I6130.}}

\date{}

\author[1]{Manuel Kauers}
\author[2]{Philipp Nuspl}
\author[2]{Veronika Pillwein}
\affil[ ]{\normalfont manuel.kauers@jku.at, philipp.nuspl@jku.at, veronika.pillwein@risc.jku.at \vspace{0.4cm}}
\affil[1]{Johannes Kepler University Linz, Institute for Algebra}
\affil[2]{Johannes Kepler University Linz, Research Institute for Symbolic Computation}

\begin{document}

\maketitle

\begin{abstract}
A sequence is called $C$-finite if it satisfies a linear recurrence with constant coefficients. We study sequences which satisfy a linear recurrence with $C$-finite coefficients. Recently, it was shown that such $C^2$-finite sequences satisfy similar closure properties as $C$-finite sequences. In particular, they form a difference ring. 

In this paper we present new techniques for performing these closure properties of $C^2$-finite sequences. These methods also allow us to derive order bounds which were not known before. Additionally, they provide more insight in the effectiveness of these computations. 

The results are based on the exponent lattice of algebraic numbers. We present an iterative algorithm which can be used to compute bases of such lattices.
\end{abstract}

\maketitle


\section{Introduction} \label{sec:introduction}
Infinite objects that can be represented by a finite amount of information and
that can be effectively computed with, e.g., by means of closure properties, are
natural objects of study in symbolic computation. This includes in particular
sequences that can be defined by linear recurrences with coefficients that, in
turn, have a finite description. If these coefficients are polynomials, the
sequences are called holonomic or $D$-finite and the special case of constant
coefficients is referred to as $C$-finite sequences.

It is well known~\cite{kauers11,mallinger96} that these form classes that are closed under
several operations such as addition, multiplication, interlacing, taking
subsequences, etc. These closure properties are algorithmic, have been
implemented in several computer algebra systems, and contribute to the
``holonomic toolkit''~\cite{kauers13} for automatically proving and deriving
identities.

It has been shown~\cite{nuspl21,nuspl21b,nuspl22b,nuspl22c} that many of these
closure properties also hold and can be implemented for sequences that are
defined by linear recurrences with $C$-finite coefficients, also called
$C^2$-finite. To our knowledge, $C^2$-finite sequences have first been
introduced formally by Kotek and Makowsky~\cite{kotek14} in the context of graph
polynomials. Thanatipanonda and Zhang~\cite{thanatipanonda20} give an overview
on different properties of polynomial, $C$-finite and holonomic sequences and
consider the extension under the name $X$-recursive sequences. A survey on closure properties of linear recurrence sequences including $C^2$-finite sequences is given by Krityakierne and Thanatipanonda~\cite{krityakierne22}.

The main computational issue when dealing with this new class of $C^2$-finite
sequences is the presence of zero divisors. Even though it was shown that
$C^2$-finite sequences form a difference ring, so far it was not clear whether
their closure properties can be effectively computed, see also the discussion in
Section~\ref{subsec:c2} below.

In this paper, we introduce a new method for executing closure properties that,
in particular, comes with order bounds.  A key ingredient for this technique is
the computation of a basis for the exponent lattice for the eigenvalues of the
coefficient recurrences. For the computation, we introduce an iterative version
of Ge's algorithm~\cite{ge93} described in Section~\ref{sec:expLattice}.  The
resulting order bounds for the ring operations, interlacing, and taking
subsequences are presented in Section~\ref{sec:orderBounds}.


\section{Preliminaries} \label{sec:preliminaries}

In this section we introduce some basic notation and definitions which are used throughout the paper. We denote the set of natural numbers by $\NN = \{ 0,1,2,\dots \}$. Furthermore, $\KN \supseteq \QN$ denotes an algebraic number field. The $\KN$-algebra of sequences under termwise addition and termwise multiplication is denoted by $\KN^\NN$. For the sake of a cleaner notation, $c(n)$ can denote both a sequence $(c(n))_{n \in \NN}$ and the term at index~$n$. The meaning is always clear from the context. The \emph{shift operator} $\shift$ acts as $\shift((a(n))_{n \in \NN}) = (a(n+1))_{n \in \NN}$ on a sequence $(a(n))_{n \in \NN}$. A difference subring $R \subseteq \KN^\NN$ is a subring which is additionally closed under taking shifts, i.e., $\shift(a) \in R$ for all $a \in R$. The ring of recurrence operators $R[\shift]$ is, in general, non-commutative and an element $\op{a} \coloneqq \sum_{i=0}^r c_i \shift^i \in R[\shift]$ with $c_i \in R$ acts on a sequence $a=(a(n))_{n \in \NN}$ as $\op{a}a  = \sum_{i=0}^r c_i(n) a(n+i)$. If $\op{a} a = 0$, we say that the operator $\op{a}$ annihilates the sequence~$a$. If $c_r \neq 0$, then~$r$ is called the \emph{order} of the operator~$\op{A}$. The minimal order of an operator which annihilates~$a$ is called the order of the sequence $a$ and is denoted by~$\ord(a)$. 

\subsection{$C$-finite sequences} \label{subseq:cfin}

Sequences $c \in \KN^\NN$ which are annihilated by an operator $\op{c} = \sum_{i=0}^r \gamma_i \shift^i \in \KN[\shift]$ are called \emph{$C$-finite}. Equivalently, these are sequences that satisfy a linear recurrence with constant coefficients 
    \[ \gamma_0 c(n) + \dots + \gamma_r c(n+r) = 0 \forAlln. \]
The set of $C$-finite sequences over~$\KN$ forms a $\KN$-algebra which we denote by $\CFin$. Suppose $c,d,c_1,\dots,c_m \in \CFin$. Then, the following \emph{closure properties} are well known (e.g., \cite{kauers11}):
\begin{enumerate}
    \item $c+d \in \CFin$ with $\ord(c+d) \leq \ord(c)+\ord(d)$,
    \item $cd \in \CFin$ with $\ord(cd) \leq \ord(c) \ord(d)$,
    \item $c_{\ell,k} \coloneqq (c(\ell n+k))_{n \in \NN} \in \CFin$ with $\ord(c_{\ell,k}) \leq \ord(c)$ for all $\ell, k \in \NN$. 
    \item Let $e$ be the interlacing of $c_1,\dots,c_m$, i.e., $e(n)=c_r(q)$ for all $n=qm+r$ with $0\leq r < m$. Then, $e \in \CFin$ and $\ord(e) \leq m \sum_{j=1}^m \ord (c_j)$. 
\end{enumerate}
The same closure properties and order bounds hold for $D$-finite sequences, i.e., sequences which are annihilated by an operator $\op{a} \in \KN[x][\shift]$ \cite{mallinger96,kauers14}.

Let $\op{c} \coloneqq \sum_{i=0}^{r-1} \gamma_i \shift^i + \shift^r$ be the unique monic minimal annihilating operator of $c \in \CFin$. The polynomial $\sum_{i=0}^{r-1} \gamma_i x^i + x^r \in \KN[x]$ is called the \emph{characteristic polynomial} of~$c$. Over the splitting field~$\LN$ the polynomial completely factors as $x^{n_0} \prod_{i=1}^m (x-\lambda_i)^{d_i}$ with pairwise different $\lambda_1,\dots,\lambda_m \in \LN$ and $n_0,d_1,\dots,d_m \in \NN$. We call these $\lambda_i$ the \emph{eigenvalues} of the sequence~$c$. The sequence can also be written as polynomial-linear combination of exponential sequences $\lambda_i^n$: In particular, there are polynomials $p_1,\dots,p_m \in \LN[x]$ with $\deg(p_i) = d_i-1$ for $i=1,\dots,m$ such that 
    \begin{align} \label{eq:closedForm}
        c(n) = \sum_{i=1}^m p_i(n) \lambda_i^n \text{ for all } n \geq n_0.
    \end{align}
This is called the \emph{closed form} of~$c$ \cite{kauers11,melczer21}. 

A $C$-finite sequence~$c$ is called \emph{degenerate} if it has eigenvalues $\lambda \neq \mu$ such that $\tfrac{\lambda}{\mu}$ is a root of unity. Otherwise, the sequence is called \emph{non-degenerate}. 

\begin{thm}[\cite{berstel76,everest15}]
    Let~$c$ be a non-degenerate $C$-finite sequence. Then,~$c$ is either the zero sequence or it only has finitely many zeros, i.e., there is an~$n_0 \in \NN$ such that $c(n) \neq 0$ for all $n \geq n_0$.
\end{thm}

Suppose~$c,d$ are $C$-finite sequences with eigenvalues~$\lambda_1,\dots,\lambda_{m_1}$ and $\mu_1,\dots,\mu_{m_2}$, respectively. From the closed form of~$c$ and~$d$, it is clear, that $c+d$ has eigenvalues~$\lambda_1,\dots,\lambda_{m_1},\mu_1,\dots,\mu_{m_2}$ and $cd$ has eigenvalues~$\lambda_i \mu_j$ with $1 \leq i \leq m_1$ and $1 \leq j \leq m_2$. The sequence $(c(\ell n+k))_{n \in \NN}$ has eigenvalues~$\lambda_1^\ell,\dots,\lambda_m^\ell$.

\subsection{$C^2$-finite sequences}\label{subsec:c2}

A generalization of $C$-finite sequences are $C^2$-finite sequences. These extend $C$-finite and $D$-finite sequences and include many more sequences which appear in combinatorics.

\begin{dfn} \label{def:C2finite}
    A sequence $a \in \KN^\NN$ is called $C^2$-finite over $\KN$ if there are $C$-finite sequences $c_0,\dots,c_r$ over $\KN$ with $c_r(n)\neq 0$ for all~$n \geq n_0$ for some $n_0 \in \NN$ such that
        \begin{align} \label{eq:C2finite}
            c_0(n) a(n) + \dots + c_r(n) a(n+r) = 0, \forAlln.
        \end{align} 
\end{dfn}

Several examples for $C^2$-finite sequences are given in~\cite{thanatipanonda20,nuspl21b}. Throughout this article, additional examples are given. 

A $C$-finite sequence~$c$ can be uniquely described by a minimal recurrence and~$\ord(c)$ many initial values. Similarly, a $C^2$-finite sequence can be described uniquely by its recurrence and by finitely many initial values. The number of initial values which is needed to uniquely determine the sequence depends on the zeros of the leading coefficient $c_r$ of the recurrence. It can be decided whether the leading coefficient only has finitely many zeros~\cite{berstel76}. However, it is not known if these finitely many zeros can be computed. This is known as the \emph{Skolem problem}~\cite{ouaknine12}. 

Previously, it was shown that $C^2$-finite sequences, analogously to $C$-finite sequences, form a $\KN$-algebra and they are furthermore closed under taking subsequences at arithmetic progressions and interlacing~\cite{nuspl21,nuspl21b}. So far, it was not known whether these closure properties can be computed effectively. In this article we show a method how these closure properties can be performed effectively. As a caveat, these computations might introduce finitely many zeros in the leading coefficient which can yield to problems when one has to decide how many initial values are needed to uniquely define the sequence. In practice we have, however, observed that even though the Skolem problem is very difficult in general it can usually be solved for most examples that appear in practice \cite{nuspl22}.

Note that in~\cite{nuspl21,nuspl21b} it is assumed that the leading coefficient~$c_r$ in~\eqref{eq:C2finite} has no zero terms at all. The two definitions are equivalent. If a sequence satisfies a $C^2$-finite recurrence as in Definition~\ref{def:C2finite}, then shifting the recurrence yields a recurrence where the leading coefficient has no zero terms. The definition here allows us to derive bounds for the orders of closure properties similar to the $C$-finite case that cannot be derived otherwise (see Example~\ref{ex:singularity}). 

If $c_r(n) \neq 0$, then the recurrence~\eqref{eq:C2finite} can be used to compute the term~$a(n+r)$ provided that the previous terms~$a(n),\dots,a(n+r-1)$ are known:
\[ a(n+r) = - \tfrac{c_0(n)}{c_r(n)} a(n) - \dots - \tfrac{c_{r-1}(n)}{c_r(n)} a(n+r-1). \]
This is also captured by the companion matrix~$M_a$ of~$a$ which is defined as 
\begin{align*}
    M_a \coloneqq 
    \begin{pmatrix}
        0 & 0 & \dots & 0 & -c_0/c_r \\
        1 & 0 & \dots & 0 & -c_1/c_r \\
        0 & 1 & \dots & 0 & -c_2/c_r \\
        \vdots & \vdots & \ddots & \vdots & \vdots \\
        0 & 0 & \dots & 1 & -c_{r-1}/c_r
    \end{pmatrix}.
\end{align*} 
If~$c_r(n) \neq 0$ for all~$n \in \NN$, then 
\[ (\shift a,\shift^2 a, \dots,\shift^r a) = (a,\shift a, \dots,\shift^{r-1} a) \, M_a. \]
In the special case that~$a$ is a $C$-finite sequence, we have $M_a \in \KN^{r \times r}$. Since the recurrence of a sequence is not unique, neither is the companion matrix. 

In the recurrence~\eqref{eq:C2finite} we can assume that $n_0=0$ holds in the closed form~\eqref{eq:closedForm} for all coefficients $c_0,\dots,c_r$. This can be achieved by extending the closed form representation to all $n \in \NN$ and introducing polynomial factors $n(n-1)\cdots (n-n_0)$ in all coefficients~$c_i$. This only increases the order of the coefficients~$c_i$ and leaves the order of the overall recurrence intact.

\subsection{Lattices}

A $\ZN$-submodule $L$ of $\ZN^m$ is called a lattice. Every lattice~$L$ admits a finite basis $v_1,\dots,v_\ell \in \ZN^m$, i.e., a set of linearly independent generators of the module~$L$. We call~$\ell$ the \emph{rank} of the lattice~$L$. 

The LLL algorithm can be used to compute a basis of ``short'' vectors for the lattice~$L$~\cite{lenstra82,cohen13}. Such a basis is called a \emph{reduced} basis. Let $L=\langle v_1,\dots,v_\ell \rangle \subseteq \ZN^m$, i.e., $L$ is the lattice generated by~$v_1,\dots,v_\ell \in \ZN^m$. Let $b_1,\dots,b_r$ be a reduced basis of~$L$ and $\bar b_1,\dots,\bar b_r$ the corresponding Gram-Schmidt vectors. The reduced basis is ``short'' in the sense that (cf. (1.7)~in~\cite{lenstra82} or Theorem~2.6.2 in~\cite{cohen13}) 
\begin{align} \label{eq:lllProperty}
    \norm{b_j}^2 \leq 2^{k-1} \norm{\bar b_k}^2 \text{ for all } 1 \leq j \leq k \leq r.
\end{align}

Suppose~$V \in \ZN^{m \times \ell}$ and $r=\min(m,\ell)$. Then, we can compute unimodular (i.e., invertible) matrices $P \in \ZN^{m \times m}, Q \in \ZN^{\ell \times \ell}$ and a diagonal matrix $D = \diag(d_1,\dots,d_r) \in \ZN^{m \times \ell}$ with $d_i \mid d_{i+1}$ for all $i=1,\dots,r-1$ such that $PVQ=D$. The unique matrix~$D$ is called the \emph{Smith normal form} of~$V$ and the largest diagonal entry~$d_r$ is called the \emph{invariant factor} of~$V$. If $e_i$ denotes the $i$-th determinantal divisor of~$V$, i.e., the greatest common divisor of all $i$-by-$i$ minors of~$V$, then $d_r=\tfrac{e_{r}}{e_{r-1}}$ \cite{newman72,middeke19}.


\section{The exponent lattice of algebraic numbers} \label{sec:expLattice}

Let $\lambda_1,\dots,\lambda_m \in \AN$. The set of relations of these algebraic
numbers
\[
L \coloneqq L(\lambda_1,\dots,\lambda_m) \coloneqq \{ (e_1,\dots,e_m) \in \ZN^m \st \lambda_1^{e_1} \cdots \lambda_m^{e_m} = 1 \}
\] 
forms a lattice.
In his PhD thesis~\cite{ge93} Ge gave an algorithm for computing a basis of~$L$.
It is a combination of LLL with a bound on the size
of the basis vectors and the fact that membership of $L$ is easy to decide.
Variants of Ge's algorithm were given in \cite{kauers05,faccin14,zheng19,zheng20,zheng21}.
Here we present another variant.
Our version is inspired by how LLL is applied in van Hoeij's algorithm for polynomial factorization~\cite{hoeij02,hoeij12,hoeij13}.
One feature of this version is that it uses approximations that are only as good as necessary for the
particular input, rather than approximations whose accuracy is determined by the worst case behavior.
Another advantage of our version is that its correctness admits a very concise proof.

Like Ge, we start by observing that 
\begin{align*}
    (e_1,\dots,e_m)\in L(\lambda_1,\dots,\lambda_m)
\iff \lambda_1^{e_1}\cdots\lambda_m^{e_m}=1 
\iff \sum_{i=1}^m e_i\log(\lambda_i) \in 2\pi\i \ZN.
\end{align*}
Hence, instead of finding a basis for $L=L(\lambda_1,\dots,\lambda_m)$, we can compute a basis of the lattice 
\begin{align*}
  L_+ = \left\{(e_1,\dots,e_{m+1})\in \ZN^m \st \sum_{i=1}^m e_i \log(\lambda_i) + e_{m+1} 2 \pi \i = 0\right\}
\end{align*}
and drop the last coordinates to find a basis of the original lattice~$L$. If we agree to always choose the standard branch of the logarithm, the last coordinate will
be bounded by $md$, where $d$ is the degree of the field extension of~$\lambda_1,\dots,\lambda_m$. By a result by Masser~\cite{masser88}, we can compute a constant~$M \geq md$ such that~$L$ and therefore~$L_+$ have a basis of vectors $b$ with $\norminf{b} \leq M$. 

It remains to provide an algorithm which can compute a basis of 
\begin{align*}
  L_+ = \left\{(e_1,\dots,e_n)\in\ \ZN^n \st e_1x_1+\cdots+e_nx_n = 0\right \}
\end{align*}
where $x_1,\dots,x_n\in \CN \setminus\{0\}$. Due to the special shape of the $x_i$ in our case, we can compute rational approximations $\xi_i \in \QN(\i)$ of arbitrary precision \cite{brent76}. In particular, for every $\epsilon>0$ we can compute $\xi_1,\dots,\xi_n\in\QN(\i)$ such that $\abs{\Re(x_i)-\Re(\xi_i)}<\epsilon$ and $\abs{\Im(x_i)-\Im(\xi_i)}<\epsilon$ for all~$i=1,\dots,n$. Furthermore, we can use the fact that membership $(e_1,\dots,e_n) \in L_+$ can be checked and that we know that a basis with vectors bounded by~$M$ exists. 

\stepcounter{dfn} 

\begin{algorithm}[h]
    \caption{Computing a basis for~$L_+$}\label{alg:basisLplus}
    \SetKwInOut{Input}{Input}
    \SetKwInOut{Assumption}{Needs}
    \SetKwInOut{Output}{Output}
    \DontPrintSemicolon

    \Input{Computable numbers $x_1,\dots,x_n\in \CN \setminus\{0\}$ and $M \in \QN$ such that the lattice \[L_+ = \left\{(e_i)_{i = 1,\dots,n}\in\ \ZN^n \st \sum_{i=1}^n e_ix_i= 0\right \}\]
    has a basis of vectors $b \in \ZN^n$ with $\norminf{b} \leq M$}
    \Assumption{One can decide whether $b \in L_+$ for any $b \in \ZN^n$}
    \Output{A basis of $L_+$}
    $w \gets 1$\;
    $B \gets \{(1,0,\dots,0,0,0),\dots, (0,\dots,0,1,0,0)\}\subseteq\ZN^{n+2}$\;
    \While{$\exists\ (e_1,\dots,e_n,\ast,\ast)\in B \colon (e_1,\dots,e_n)\not\in L_+$}{ \nllabel{alg:loopCondition}
        $w \gets 2w$ \;
        find $\xi_1,\dots,\xi_n\in\QN(\i)$ with $\abs{\Re(\xi_i)-\Re(x_i)}<\frac1{nw}$ and $\abs{\Im(\xi_i)-\Im(x_i)}<\frac1{nw}$ for all~$i=1,\dots,n$\;
        replace every vector $(e_1,\dots,e_n,\ast,\ast)\in B$ by \[\left(e_1,\dots,e_n, w \sum_{i=1}^n e_i\Re(\xi_i), w \sum_{i=1}^n e_i\Im(\xi_i)\right)\]\;
        apply LLL to $B$, call the output vectors $b_1,\dots,b_r$ and the corresponding Gram-Schmidt vectors $\bar b_1,\dots,\bar b_r$ \; 
        \lWhile{$r>0$ and $\norm{\bar b_r}>\sqrt{n+2}M$}{$r \gets r-1$}
        $B \gets \{b_1,\dots,b_r\}$ \nllabel{alg:removeVectors} \;
    }
    \Return{$\{(e_1,\dots,e_n):(e_1,\dots,e_n,\ast,\ast)\in B\}$}
\end{algorithm}

For proving the correctness of Algorithm~\ref{alg:basisLplus} we will employ the following lemma:

\begin{lem} \label{lemma:gramSchmidt} \cite[Lemma~2]{hoeij12}
If $b_1,\dots,b_r$ is a lattice basis and $\bar b_1,\dots,\bar b_r$ is the corresponding Gram-Schmidt basis,
then for every $v\in\langle b_1,\dots,b_r \rangle$ with $\norm{v}<\norm{\bar b_r}$ we have in fact $v\in\langle b_1,\dots,b_{r-1} \rangle$.
\end{lem}

\begin{proof}
Let $v\in\langle b_1,\dots,b_r \rangle$ be such that $\norm{v}<\norm{\bar b_r}$, say
$v=\alpha_1b_1+\cdots+\alpha_rb_r$ for certain $\alpha_1,\dots,\alpha_r\in\ZN$.
We have to show that $\alpha_r=0$. Let $\mu_{i,j}$ be such that
$b_i=\sum_{j\leq i}\mu_{i,j}\bar b_j$ for all $i,j$; note that $\mu_{i,i}=1$. 
Now
\begin{alignat*}1
  \norm{\bar b_r}^2&>\norm{v}^2=\norm{\sum_{i=1}^r\alpha_ib_i}^2=\norm{\sum_{i=1}^r\sum_{j=1}^i\alpha_i\mu_{i,j}\bar b_j}^2\\
  &=\norm{\sum_{j=1}^r\biggl(\sum_{i=j}^r\alpha_i\mu_{i,j}\biggr)\bar b_j}^2
  \overset{\text{Pythagoras}}=\sum_{j=1}^r\abs{\sum_{i=j}^r\alpha_i\mu_{i,j}}^2\,\norm{\bar b_j}^2\\
  &=\underbrace{\sum_{j=1}^{r-1}\abs{\sum_{i=j}^r\alpha_i\mu_{i,j}}^2\,\norm{\bar b_j}^2}_{\geq0}
  + \abs{\alpha_r}^2\,\norm{\bar b_r}^2\geq \abs{\alpha_r}^2\,\norm{\bar b_r}^2
\end{alignat*}
together with $\alpha_r\in\ZN$ forces $\alpha_r=0$, as claimed.
\end{proof}

\begin{thm} \label{theorem:correctnessAlgorithm}
Algorithm~\ref{alg:basisLplus} is correct and terminates.
\end{thm}

\begin{proof}
    It is clear that every output vector is an element of~$L_+$.
    To see that the output vectors generate~$L_+$, we need to justify the removals in line~\ref{alg:removeVectors}.
    By assumption, we know that $L_+$ has a basis whose elements have components bounded by~$M$.
    For every vector $(e_1,\dots,e_n)\in L_+$ with $\abs{e_i}<M$ for all $i$ we have
    \begin{align*}
        w\abs{\sum_{i=1}^n e_i\Re(\xi_i)}
        &=w\abs{\sum_{i=1}^n e_i\Re(\xi_i)-e_i\Re(x_i)} 
        \leq w\sum_{i=1}^n \abs{e_i}\abs{\Re(\xi_i)-\Re(x_i)} \\
        &\leq w\sum_{i=1}^n \abs{e_i}/(nw) < M
    \end{align*}
    and likewise for the imaginary parts.
    Therefore, we are only interested in vectors $b=(e_1,\dots,e_n,\ast,\ast)$
    in the lattice generated by $B$ with
    \[
      \norm{b} \leq \sqrt{M^2+\cdots+M^2+M^2+M^2}=\sqrt{n+2}M.
    \]
    By Lemma~\ref{lemma:gramSchmidt}, these vectors are still in the lattice after the removals in line~\ref{alg:removeVectors}.
    
    It remains to show that the algorithm terminates. Suppose it does not terminate, i.e., the set~$B$ eventually contains~$r$ vectors in every iteration which are not all in the lattice~$L_+$. We show that from some point on in the algorithm (i.e., for big enough~$w$), this cannot be the case because vectors which are not in~$L_+$ are too long and are therefore removed in line~\ref{alg:removeVectors} of the algorithm.
    
    There are only finitely many vectors $(e_1,\dots,e_n)\in\ZN^n$ with $\abs{e_i}\leq \sqrt{n+2}M$ for all~$i=1,\dots,n$.
    Therefore, there exists an~$\epsilon > 0$ such that
    \[
     \max(\abs{e_1\Re(x_1)+\cdots+e_n\Re(x_n)},\abs{e_1\Im(x_1)+\cdots+e_n\Im(x_n)})>\epsilon
    \]
    for all $(e_1,\dots,e_n)\in\ZN^n\setminus L_+$ with $\abs{e_i}\leq \sqrt{n+2}M$ for all~$i$.
    Choose such an~$\epsilon$. Suppose we are in line~\ref{alg:loopCondition} of the algorithm with $w \geq \tfrac{\sqrt{n+2}M (1+2^{(r-1)/2})}{\epsilon}$ and $b_j \in B \setminus L$ with $j \in \{1,\dots,r\}$. Let 
      \[ b_j=\left(e_1,\dots,e_n, w \sum_{i=1}^n e_i\Re(\xi_i), w \sum_{i=1}^n e_i\Im(\xi_i)\right). \] 
    Since $b_j$ has not been removed in line~\ref{alg:removeVectors} in the previous iteration, we have $\abs{e_i}\leq \sqrt{n+2}M$ for all $i=1,\dots,n$. By the choice of $\epsilon$ for either $f=\Re$ or $f=\Im$ we have
    \begin{align*}
      w\abs{\sum_{i=1}^n e_i f(\xi_i)}
      &=w\abs{\sum_{i=1}^n e_i (f(\xi_i) - f(x_i) + f(x_i))} 
      =w\abs{\sum_{i=1}^n e_i (f(\xi_i) - f(x_i)) + \sum_{i=1}^n e_if(x_i)}
      \\
      &\geq w \left(\abs{\sum_{i=1}^n e_if(x_i)}-\abs{\sum_{i=1}^n e_i (f(\xi_i) - f(x_i))}\right) 
      > w \left(\epsilon-\abs{\sum_{i=1}^n e_i (f(\xi_i) - f(x_i))}\right).
    \end{align*}
    Furthermore, we have 
    \begin{align*}
      \abs{\sum_{i=1}^n e_i (f(\xi_i) - f(x_i))} \leq \sum_{i=1}^n \abs{e_i} \abs{f(\xi_i) - f(x_i)} 
      < \sum_{i=1}^n \sqrt{n+2}M \tfrac{1}{nw} = \tfrac{\sqrt{n+2}M}{w}.
    \end{align*}
    Using this and the condition on~$w$ in the inequality above yields
    \begin{align*}
      w\abs{\sum_{i=1}^n e_i f(\xi_i)} &> we - \sqrt{n+2}M 
      \geq 2^{(r-1)/2} \sqrt{n+2}M.
    \end{align*}
    Therefore, $\norm{b_j} > 2^{(r-1)/2} \sqrt{n+2}M$. In particular, using~\eqref{eq:lllProperty},
    \begin{align*}
      \norm{\bar b_k} \geq 2^{-(k-1)/2}\norm{b_j} > 2^{(r-k)/2} \sqrt{n+2}M \geq \sqrt{n+2}M
    \end{align*}
    for all $k=j,\dots,r$. Hence, $b_j,\dots,b_r$ would have been removed already in the past iteration and cannot be in the set~$B$ anymore, a contradiction.    
\end{proof}

An implementation of the algorithm is part of the \texttt{rec\_sequences} package\footnote{The code is available at \url{https://github.com/PhilippNuspl/rec_sequences} in the \texttt{IntegerRelations} class.} and is publicly available \cite{nuspl22c}.


\section{Torsion number} \label{sec:torsionNumber}

For proving the order bounds for $C^2$-finite sequences, we will heavily rely on the fact that a $C$-finite sequence~$c$ can be written as interlacing of non-degenerate sequences~$c(dn),\dots,c(dn+d-1)$ \cite[Theorem~1.2]{everest15}. More generally, if we have a finitely generated difference algebra of $C$-finite sequences, we will determine a number~$d \in \NN$ (which we will call the \emph{torsion number}) such that every sequence in the algebra can be written as the interlacing of~$d$ non-degenerate subsequences. 

Let~$c_0,\dots,c_r \in \CFin$ with eigenvalues $\lambda_1,\dots,\lambda_m$ and let $R_d \coloneqq \diffalg{\KN}{c_0(dn),\dots,c_r(dn)}$ be the smallest difference algebra which contains the sequences $c_0(dn),\dots,c_r(dn)$. Suppose $c \in R_d$. Then, every eigenvalue $\lambda$ of~$c$ is of the form $\lambda=\lambda_1^{e_1}\cdots \lambda_m^{e_m}$ for some $e_1,\dots,e_m \in \NN$. We want to find a~$d$ such that every sequence $c \in R_d$ is non-degenerate. Equivalently, we want to find a~$d$ such that 
    \[ \left( \tfrac{\lambda_1^{d e_1}\cdots \lambda_m^{d e_m}}{\lambda_1^{d f_1}\cdots \lambda_m^{d f_m}} \right)^k = 1 \implies \lambda_1^{d e_1}\cdots \lambda_m^{d e_m}=\lambda_1^{d f_1}\cdots \lambda_m^{d f_m} \]
for all $k,e_1,\dots,e_m,f_1,\dots,f_m \in \NN$. In order to write this more concisely we define the multiplicative group $G \coloneqq \langle \lambda_1, \dots, \lambda_m \rangle \leq (\CN^\times, \cdot)$. Then, this condition reads as 
    \[  \forall k \in \NN_{\geq 1} \forall \lambda \in G \colon \lambda^{kd} = 1 \implies \lambda^d = 1.\]
The following lemma shows that this number~$d$ also has a purely group-theoretical and a purely lattice-theoretical description.

\begin{lem} \label{lemma:equivalencesTorisonNumber}
    Let $G \coloneqq \langle \lambda_1, \dots, \lambda_m \rangle \leq (\CN^\times, \cdot)$. The following conditions on~$d \in \NN_{\geq 1}$ are equivalent:
    \begin{enumerate}
        \item \label{enum:equivalencesTorisonNumber1} The number~$d$ satisfies 
            \[ \forall k \in \NN_{\geq 1} \forall \lambda \in G \colon \lambda^{kd} = 1 \implies \lambda^d = 1. \]
        \item \label{enum:equivalencesTorisonNumber2} Let $\torsion{G} \coloneqq \{ \lambda \in G \st \ord(\lambda) < \infty \}$ be the torsion subgroup of~$G$. Then, $d$ satisfies 
            \[ \ord(\lambda) \mid d \text{ for all } \lambda \in \torsion{G}. \]
        \item \label{enum:equivalencesTorisonNumber3} Let 
            \[ L \coloneqq L(\lambda_1,\dots,\lambda_m) \coloneqq \{ (e_1,\dots,e_m) \in \ZN^m \st \lambda_1^{e_1} \cdots \lambda_m^{e_m} = 1 \} \] 
        be the lattice of integer relations among $\lambda_1,\dots,\lambda_m$. Then, $d$ satisfies 
            \begin{align} \label{eq:torsionNumberLattice}
                \forall k \in \NN_{\geq 1} \, \forall v \in \ZN^m \colon k d v \in L \implies dv \in L.
            \end{align}
    \end{enumerate}
\end{lem}

\begin{proof}
    \ref{enum:equivalencesTorisonNumber1} $\implies$ \ref{enum:equivalencesTorisonNumber2}: Let $\lambda \in \torsion{G}$ and let $m \in \NN_{\geq 1}$ be minimal with $\lambda^m=1$. Then, clearly $\lambda^{md}=1$. By assumption, $\lambda^d=1$. As $m$ was chosen minimal, we have $m \mid d$.

    \ref{enum:equivalencesTorisonNumber2} $\implies$ \ref{enum:equivalencesTorisonNumber3}: Let $k \in \NN_{\geq 1}, v=(e_1,\dots,e_m) \in \ZN^m$ and $kdv \in L$. Let $\lambda=\lambda_1^{e_1} \cdots \lambda_m^{e_m}$. By definition of~$L$,
    \[ \lambda^{kd} = \lambda_1^{kd e_1} \cdots \lambda_m^{kd e_m} = 1. \]
    Hence, $\lambda \in \torsion{G}$. Therefore, by assumption, $\ord(\lambda) \mid d$, so $\lambda^d = 1$ and $dv \in L$.

    \ref{enum:equivalencesTorisonNumber3} $\implies$ \ref{enum:equivalencesTorisonNumber1}: Let $k \in \NN_{\geq 1}, \lambda = \lambda_1^{e_1}\cdots \lambda_m^{e_m} \in G$ and $\lambda^{kd} = 1$. Defining $v \coloneqq (e_1,\dots,e_m) \in \ZN^m$ yields $kd v \in L$. By assumption, $dv \in L$, i.e., $\lambda^d = 1$.
\end{proof}

Considering condition~\ref{enum:equivalencesTorisonNumber2} of Lemma~\ref{lemma:equivalencesTorisonNumber}, we can see that the smallest~$d$ which satisfies the condition is the exponent of the torsion group.

\begin{dfn}
    The \emph{torsion number}~$d \in \NN_{\geq 1}$ of $\lambda_1,\dots,\lambda_m \in \AN$ is defined as 
    \[ d \coloneqq \exp (\torsion{G}) \coloneqq \lcm(\ord(\lambda) \st \lambda \in \torsion{G}) \]
    where $G \coloneqq \langle \lambda_1, \dots, \lambda_m \rangle \leq (\CN^\times, \cdot)$. 
\end{dfn}

We also call~$d$ the torsion number of the lattice~$L$ if it is the smallest number satisfying~\eqref{eq:torsionNumberLattice}. Using the terminology of pure modules,~\eqref{eq:torsionNumberLattice}~is equivalent to $\{ v\in \ZN^m \st dv \in L \}$ being a pure lattice \cite[III.16]{curtis66}. 

\begin{lem} \label{lemma:divsibilityTorsionNumber}
    Let~$d$ be the torsion number of the lattice $L$. Then,~$d'$ satisfies the conditions from Lemma~\ref{lemma:equivalencesTorisonNumber} if and only if~$d \mid d'$.
\end{lem}

\begin{proof}
    $\implies$: By definition,~$d$ is the smallest number satisfying the conditions from Lemma~\ref{lemma:equivalencesTorisonNumber}. Hence, $d'>d$ and we can write $d'=dq+r$ with $0 \leq r < d$. Let $G \coloneqq \langle \lambda_1, \dots, \lambda_m \rangle$ and $\lambda \in \torsion{G}$. Then, $\ord(\lambda) \mid d$ and $\ord(\lambda) \mid d'$. Hence, 
    \[ \ord(\lambda) \mid (d'-dq) = r. \]
    Since $d>r$ is the smallest number with this property we have $r=0$, so $d \mid d'$.

    $\Longleftarrow$: Clear from condition~\ref{enum:equivalencesTorisonNumber2} of Lemma~\ref{lemma:equivalencesTorisonNumber}.
\end{proof}

Now, we want to show that the torsion number of algebraic numbers~$\lambda_1,\dots,\lambda_m \in \AN$ can actually be computed. First, we have devised an algorithm in Section~\ref{sec:expLattice} which computes a basis~$v_1,\dots,v_\ell \in \ZN^m$ for the lattice $L \coloneqq L(\lambda_1,\dots,\lambda_m)$. Then, the invariant factor of the matrix built by the basis is precisely the torsion number of the lattice:

\begin{thm} \label{theorem:invariantFactor}
    Let $v_1,\dots,v_\ell \in \ZN^m$ be a basis of the lattice $L = \langle v_1, \dots, v_\ell \rangle \subseteq \ZN^m$. Let $V \coloneqq (v_1,\dots,v_\ell) \in \ZN^{m \times \ell}$. Then, the invariant factor of~$V$ is the torsion number of~$L$.
\end{thm}

\begin{proof}
    We write 
    \[ PVQ=\begin{pmatrix} D \\ 0 \end{pmatrix}  \eqqcolon \overline{D} \in \ZN^{m \times \ell} \]
    where $D = \diag(d_1,\dots,d_{\ell-1},d)$ is the Smith normal form and $P \in \ZN^{m\times m},Q \in \ZN^{\ell \times \ell}$ are unimodular matrices.

    First, we show that the invariant factor~$d$ of the matrix~$V$ satisfies~\eqref{eq:torsionNumberLattice}: Let $k \in \NN_{\geq 1}$ and $v \in \ZN^m$ with $kdv \in L$. Then, there is a $w=(w_1,\dots,w_\ell) \in \ZN^\ell$ such that 
    \[ kdv = w_1 v_1 + \dots + w_\ell v_\ell = Vw. \]
    Therefore,
    \[ kdPv = PVw = PVQQ^{-1}w = \overline{D} \overline{w} \]
    with $\overline{w} = Q^{-1}w$. The $\ell$-th row yields $kd(Pv)_\ell=d \overline{w}_\ell$ where $(Pv)_\ell$ denotes the $\ell$-th entry of the vector $Pv$. Hence, $k \mid \overline{w}_\ell$. For the $i$-th row with $i<\ell$ we have $kd (Pv)_i=d_i \overline{w}_i$. By the property of the Smith normal form, we have $d_i \mid d$ and therefore $k \mid \overline{w}_i$. As~$Q$ is unimodular, \cite[Corollary 158]{middeke19} yields 
    \[ k \mid \gcd(\overline{w}) = \gcd(Q^{-1}w) = \gcd(w).\]
    Hence, $\tfrac{w}{k} \in \ZN^\ell$ and 
    \[ dv = \tfrac{1}{k} k dv = V \tfrac{w}{k} \in L. \]

    Secondly, we show that the invariant factor~$d$ is the smallest number: Suppose~$d'$ is the smallest number which satisfies~\eqref{eq:torsionNumberLattice}. By Lemma~\ref{lemma:divsibilityTorsionNumber} there is a~$k$ such that $d=d'k$. Let $P^{-1} = (p_1,\dots,p_m) \in \ZN^{m \times m}$. As the columns of~$V$ are a basis of $L$ and $Q$ is unimodular we have 
    \[ L = V\ZN^{\ell} = P^{-1} \overline{D} Q^{-1} \ZN^\ell = P^{-1} \overline{D} \ZN^\ell. \]
    Therefore, 
        \[ \{ d_1p_1,\dots,d_{\ell-1} p_{\ell-1},dp_\ell \} \] 
    is also a basis of~$L$. Let $v \coloneqq dp_\ell=d'kp_\ell \in L$. By assumption, $d' p_\ell \in L$. Hence, 
        \[ \{ d_1p_1,\dots,d_{\ell-1} p_{\ell-1},d' p_\ell \} \] 
    is a basis of~$L$ as well. Therefore, there is a unimodular change-of-basis matrix $U \in \ZN^{m \times m}$ with 
    \[ U(d_1p_1,\dots,d_{\ell-1} p_{\ell-1},dp_\ell) = (d_1p_1,\dots,d_{\ell-1} p_{\ell-1},d' p_\ell). \]
    In particular, the last column yields $U d p_\ell = d' p_\ell$. As~$U$ is unimodular, we have 
    \[ d \gcd(p_\ell) = \gcd(U d p_\ell) = \gcd(d' p_\ell) = d' \gcd(p_\ell). \]
    As $\gcd(p_\ell) \neq 0$, we have $d = d'$.
\end{proof}

\begin{ex}
	Let 
	\[ \lambda_1 = 2^{1/2}, \lambda_2 = (-2)^{1/3}, \lambda_3 = \i, \lambda_4 = -\i. \]
	The columns of
    \begin{align*}
        V \coloneqq \begin{pmatrix}
            0 & 0 & -2 \\
            0 & 0 & 3 \\
            1 & 2 & -1 \\
            1 & -2 & 1
        \end{pmatrix} 
        = P^{-1}
        \left(
        \begin{array}{ccc}
        1 & 0 & 0 \\
        0 & 1 & 0 \\
        0 & 0 & 4 \\
        0 & 0 & 0 \\
        \end{array}
        \right)
        Q^{-1}
    \end{align*}
	are a basis	of $L(\lambda_1,\lambda_2,\lambda_3,\lambda_4)$. Hence, $d=4$ is the torsion number of $\lambda_1,\dots,\lambda_4$. 
\end{ex}

Let~$c_0,\dots,c_r \in \CFin$ with eigenvalues $\lambda_1,\dots,\lambda_m$. Then, we have seen that we can compute a number~$d \in \NN_{\geq 1}$ (namely the torsion number) such that the algebra 
\[ R \coloneqq \diffalg{\KN}{c_0(dn),\dots,c_r(dn)} \]
only contains sequences which are non-degenerate, i.e., sequences which contain only finitely many zeros. A non-degenerate sequence might still be a zero divisor in the ring~$\KN^\NN$. However, we can still define the localization~$\fractions{R}\coloneqq \{\tfrac{c}{d} \st c \in R, d \in R \setminus \{ 0 \}\}$. This localization~$\fractions{R}$ is a field. An element of~$\fractions{R}$ can, however, only be interpreted as a sequence in~$\KN^\NN$ from some term on (cf. the discussion in Section~8.2 in~\cite{petkovsek96} or \cite{schneider20}). For instance, the sequence~$\tfrac{3^n}{2^n-1}$ cannot be evaluated at the term $n=0$. This is not a problem for our applications as we will see in Section~\ref{sec:orderBounds}. We summarize the discussions of the section in the following theorem:

\begin{thm} \label{theorem:field}
    Let~$c_0,\dots,c_r \in \CFin$ with eigenvalues $\lambda_1,\dots,\lambda_m$. Then, we can compute a number~$d \in \NN_{\geq 1}$ (namely the torsion number) such that the localization~$\fractions{R}$ of the algebra 
        \[ R \coloneqq \diffalg{\KN}{c_0(dn),\dots,c_r(dn)} \]
    is a field. The elements of the field~$\fractions{R}$ can be considered as sequences which are non-zero from some term on.
\end{thm}

From the closed form of $C$-finite sequences it is clear that these sequences can be seen as special cases of sums of single nested product expressions. The torsion number can be used to find a certain algebraic independent basis of these sequences \cite{schneider20}. 


\section{Order bounds} \label{sec:orderBounds}

In this section we will derive order bounds for the ring operations and additional closure properties of $C^2$-finite sequences. 

The computation of closure properties of $C^2$-finite sequences can be reduced to solving linear systems of equations \cite{nuspl21,nuspl21b}. A $C^2$-finite recurrence 
    \[ x_0(n)  + x_1(n) \shift + \dots + x_s(n) \shift^s  \]
with $x_i \in R$ for some suitable ring $R$ of sequences is obtained by computing an element $(x_0,\dots,x_{s})$ in the kernel of a matrix 
    \begin{align} \label{eq:matrixAnsatz}
        \begin{pmatrix}
            w_0 , w_1 , \dots, w_s
        \end{pmatrix} \in \fractions{R}^{r \times (s+1)}.
    \end{align}
The $w_i$ can be computed iteratively using $w_{i+1} = M \shift(w_i)$ for a suitable matrix $M \in \fractions{R}^{r \times r}$ (where the shift operator $\shift$ is applied componentwise). 
\begin{itemize}
    \item In the case a recurrence for $a+b$ is computed, we use $w_0 = e_0 \oplus \tilde{e}_0$ and $M=M_a \oplus M_b$ where $M_a,M_b$ are the companion matrices of $a,b$ and $e_0,\tilde{e}_0$ are the first unit vectors of appropriate sizes.
    \item In the case a recurrence for $ab$ is computed, we use $w_0 = e_0 \otimes \tilde{e}_0$ and $M=M_a \otimes M_b$.
    \item In the case a recurrence for $a(\ell n)$ is computed, we use $w_0=e_0$ and $M=M_a(\ell n)\cdots M_a(\ell n+\ell-1)$.
    \item In the case a $C^2$-finite recurrence for $c(jn^2+kn+\ell)$ with $j,k,\ell \in \NN$ and a $C$-finite sequence~$c$ (which does not have $0$ as an eigenvalue) is computed, we use 
    \begin{align} \label{eq:constrSparseSubsequence}
        w_0 = M_c^{kn+\ell-r+1} e_{r-1} \text{ and } M=M_c^{j(2n+1)}
    \end{align}
    where $M_c$ is the companion matrix of $c$ and $e_{r-1}$ the last unit vector. 
\end{itemize}
The underlying ring $R$ is the difference algebra $\diffalg{\KN}{c_0,\dots,c_r}$ generated by the $C$-finite sequences $c_0,\dots,c_r$ appearing in~$w_0$ and~$M$.

\subsection{Interlacing and subsequence}

\begin{thm} \label{theorem:interlacing}
    Let $a_1(n),\dots,a_d(n)$ be $C^2$-finite sequences of maximal order~$r$. Let $b$ be the interlacing of these sequences. We can compute a $C^2$-finite recurrence of order at most~$dr$ for~$b$.
\end{thm}

\begin{proof}
    By shifting the recurrences of the $a_s$ appropriately, we can assume that they all satisfy a $C^2$-finite recurrence of order $r$ of the form
    \[ c_{s,0}(n) a_s(n) + \dots + c_{s,r}(n) a_s(n+r) = 0  \]
    for $s=1,\dots,d$ for $C$-finite sequences $c_{s,i}$ where the $c_{s,r}$ only have finitely many zeros. Let $e_{di}$ be the interlacing of $c_{1,i},\dots,c_{d,i}$ for $i=0,\dots,r$. These $e_{di}$ are then $C$-finite and $e_{dr}$ only has finitely many zeros. Then, $b$ satisfies the recurrence 
    \[ e_0(n) b(n) + e_d(n) b(n+d) + \dots + e_{dr}(n) b(n+dr) = 0. \]
\end{proof}

As seen in the proof of Theorem~\ref{theorem:interlacing}, computing the interlacing of $C^2$-finite sequences is simpler than in the case of $C$-finite and $D$-finite sequences. This is because the coefficients of the recurrence, namely $C$-finite sequences, are closed under interlacing themselves. 

\begin{ex} \label{ex:interlacing}
    Let $c$ be $C$-finite satisfying 
    \[ c(n)-c(n+r)=0, \quad c(0)=1,c(1)=\dots=c(r-1)=0. \]
    Furthermore, let~$a$ be the interlacing of~$c$ and $d-1$ times the zero sequence. Theorem~\ref{theorem:interlacing} shows that~$a$ is $C^2$-finite of order at most~$dr$. The sequence~$a$ is cyclic and has~$rd-1$ consecutive zeros. Hence, the sequence~$a$ also has to have order at least $rd$ as otherwise, the sequence would be constantly zero. The bound in Theorem~\ref*{theorem:interlacing} is therefore tight in general.
\end{ex}

\begin{lem} \label{lemma:subsequenceD}
    Let $a$ be $C^2$-finite of order $r$ and let~$d$ be the torsion number of the eigenvalues appearing in the recurrence of~$a$. Let $\ell \in \NN$. We can compute a $C^2$-finite recurrence of order at most $r$ which is satisfied by all sequences $a(d \ell n+i)$ for $i=0,\dots,d \ell -1$.
\end{lem}

\begin{proof}
    The sequences $a(n+i)$ for $i=0,\dots,d-1$ all satisfy the same recurrence. By the choice of $d$, all sequences in the ring $R$ generated by the sequences appearing in 
    \[ M=M_a(d\ell n)\cdots M_a(d\ell n+d\ell-1) \] 
    are non-degenerate. By Theorem~\ref{theorem:field}, $\fractions{R}$ is a field. Therefore, if $s=r$, then the linear system~\eqref{eq:matrixAnsatz} is underdetermined and we can compute an element (after clearing denominators) $(x_0,\dots,x_r) \in R^{r+1}$ in the kernel with $x_t \neq 0$ and $x_{t+1}=\dots=x_r=0$ for some $t \leq r$. This gives rise to a $C^2$-finite recurrence 
    \[ x_0(n) + x_1(n) \shift + \dots + x_t(n) \shift^t \]
    as $x_t$ only has finitely many zeros by the choice of~$d$.
\end{proof}

To extend Lemma~\ref{lemma:subsequenceD} to subsequences at arbitrary arithmetic progressions we write such an arbitrary subsequence as the interlacing of certain subsequences for which Lemma~\ref{lemma:subsequenceD} can be applied.

\begin{thm} \label{thm:subsequence}
    Let $a$ be $C^2$-finite of order $r$ and let~$d$ be the torsion number of the eigenvalues appearing in the recurrence of~$a$. Let $\ell \in \NN$. We can compute a $C^2$-finite recurrence of order at most $dr$ which is satisfied by the sequence $a(\ell n)$.
\end{thm}

\begin{proof}
    By Lemma~\ref{lemma:subsequenceD} we can compute a recurrence of order at most $r$ satisfied by $a(d \ell n+i)$ for $i=0,\dots,d \ell-1$. Let $b$ be the interlacing of the~$d$ sequences 
    \[ a(d\ell n),a(d \ell n+\ell),\dots,a(d\ell n+(d-1)\ell). \]
    By Theorem~\ref{theorem:interlacing}, $b$ has order at most $dr$. We show that $b(n)=a(\ell n)$: Let $n=qd+s$ with $0\leq s < d$. Then, by the definition of $b$
    \[ b(n) = b(qd+s) = a(d\ell q+sl) = a(\ell (dq+s)) = a(\ell n). \] 
\end{proof}

\subsection{Ring operations}

\begin{thm} \label{thm:boundsRing}
    Let $a,b$ be $C^2$-finite of order $r_1,r_2$, respectively and let~$d$ be the torsion number of the eigenvalues appearing in the recurrences of~$a,b$. Then,
    \begin{enumerate}
        \item the sequence $a+b$ is $C^2$-finite of order at most $d(r_1+r_2)$ and
        \item the sequence $ab$ is $C^2$-finite of order at most $d r_1 r_2$.
    \end{enumerate}
    Furthermore, such recurrences can be computed.
\end{thm}

\begin{proof}
    We can compute $C^2$-finite recurrences of maximal order $r_1,r_2$ for $a(dn+i),b(dn+i)$ by Lemma~\ref{lemma:subsequenceD}. The closure properties $a(dn+i)+b(dn+i)$ and $a(dn+i) b(dn+i)$ can be computed again by solving a linear system of equations over the field $\fractions{R}$. Then, the same order bounds as in the $C$-finite and $D$-finite case apply, so the sequences $a(dn+i)+b(dn+i), a(dn+i) b(dn+i)$ have maximal orders $r_1+r_2,r_1 r_2$, respectively. By Theorem~\ref{theorem:interlacing}, we can interlace these sequence and obtain a recurrence of order $d(r_1+r_2),dr_1r_2$ for $a+b$ and $ab$, respectively.
\end{proof}

In the special case that both $C^2$-finite sequences are $C$-finite or $D$-finite, the torsion number is~$1$ and the bounds simplify to the known order bounds for these rings.

Theorem~\ref{thm:boundsRing} does not imply that the ring of $C^2$-finite sequences is computable. We can compute $C^2$-finite recurrences for the sum and the product. These recurrences, however, have leading coefficients which can have finitely many zeros. To uniquely determine the sequences $a+b,ab$ we might need to define additional initial values at these singularities. However, by the Skolem problem, we do not know whether these singularities can be computed. This is also illustrated in the next example.

\begin{ex} \label{ex:singularity}
    Let $a(n)=2^{\binom{n+1}{2}}$ (\href{http://oeis.org/A006125}{A006125} in the OEIS \cite{oeis}) and $b(n)=4^{\binom{n}{2}}$ (\href{http://oeis.org/A053763}{A053763}). Both sequences are $C^2$-finite satisfying the recurrences 
    \[ 2^{n+1} \, a(n) - a(n+1) = 0, \quad 4^n \, b(n) - b(n+1) = 0. \]
    The torsion number of $L(1,2,4)$ is~$d=1$. The coefficients for a recurrence of $c=a+b$ are given by an element in the kernel of 
    \[ \begin{pmatrix}
        1 & 2^{n+1} & 2^{2n+3} \\
        1 & 2^{2n} & 2^{4n+2}
    \end{pmatrix}. \]
    A recurrence is therefore, for instance, given by 
    \[ 2^{3n+3}(2^n-1) c(n) - 2^{n+2}(2^{2n}-2) c(n+1) + (2^n-2) c(n+2) = 0. \]
    The recurrence has order $\ord(a)+\ord(b)=2$ as expected but the leading coefficient has a zero term at $n=1$. Shifting the recurrence yields a recurrence of higher order with a leading coefficient which does not have any zero terms anymore.
\end{ex}

\begin{ex}
    Let $c$ be $C$-finite of order~$2$ satisfying 
    \[ c(n) - c(n+2) = 0,\quad c(0)=-1,c(1)=1. \]
    Let $a,b$ be $C^2$-finite satisfying 
    \[ a(n)=1 \quad c(n) b(n)-b(n+1) = 0, \quad b(0)=1. \]
    The eigenvalues that appear are $1$ and $-1$. The torsion number is therefore $d=2$. Let $a_i(n)=a(2n+i)$ and $b_i(n)=b(2n+i)$ for $i=0,1$. These are even $C$-finite of order~$1$ satisfying 
    \[ a_i(n)-a_i(n+1)=0, \quad b_i(n)+b_i(n+1)=0. \] Let $s_i=a_i+b_i$. These $s_i$ are $C$-finite of order~$2$ satisfying 
    \[ s_i(n)-s_i(n+2) = 0.\]
    The interlacing $s=a+b$ of $s_0,s_1$ satisfies the $C$-finite recurrence of order $4=d (\ord(a)+\ord(b))$
    \[ s(n)-s(n+4)=0. \]
    However, $s$ also satisfies a $C^2$-finite recurrence of order $3$, namely
    \[ c_0(n) s(n) + c_2(n) s(n+2) + s(n+3)=0\] 
    with 
    \begin{align*}
        c_0(n)-c_0(n+2)&=0, & c_0(0)&=-1, & c_0(1)&=0, \\
        c_2(n)-c_2(n+2)&=0, & c_2(0)&=0, & c_2(1)&=-1.
    \end{align*}
    There cannot be a shorter recurrence for $s(n)$ as it contains $2$ consecutive zeros.
\end{ex}

\subsection{Sparse subsequences}

\begin{thm} \label{thm:boundsSparseCFin}
    Let $c$ be $C$-finite of order $r$ and $\lambda_1,\dots,\lambda_m$ its eigenvalues and $\lambda_i \neq 0$ for all $i=1,\dots,m$. Let~$d$ be the torsion number of the eigenvalues. Then, we can compute a $C^2$-finite recurrence of
    \[ c(j n^2+kn+\ell) \]
    of maximal order $dr$ for all $j,k,\ell\in \NN$.
\end{thm}

\begin{proof}
    In a first step, we show how we can find a recurrence of order $r$ for the sequence 
    \[ a(n)=c(d(jn^2+kn)+\ell). \]
    Lemma~11 in \cite{kotek14} shows that $M^{pn+q}$ for $p,q \in \ZN$ is a matrix of $C$-finite sequences. The proof shows that the characteristic polynomials of the sequences is the characteristic polynomial of~$M^p$. Let $M_c$ be the companion matrix of~$c$. Suppose 
    \[ (x-\lambda_1)^{d_1} \cdots (x-\lambda_m)^{d_m} \]
    is the characteristic polynomial of $c$ which, by definition of the companion matrix, is also equal to the characteristic polynomial of~$M_c$. Then, by the closed form of $C$-finite sequences, the characteristic polynomial of $c(pn)$ is given by 
    \[ (x-\lambda_1^p)^{d_1} \cdots (x-\lambda_m^p)^{d_m} \]
    which, in turn, is equal to the characteristic polynomial of $M_c^p$. By~\eqref{eq:constrSparseSubsequence}, the sequences that generate the underlying ring $R$ used for computing a recurrence for $a(n)$ all have characteristic polynomial equal to the characteristic polynomials of $M_c^{dk}$ and $M_c^{2dj}$. An element in the kernel of the linear system over the field~$\fractions{R}$ can easily be computed if $s=r$. This gives rise to a $C^2$-finite recurrence of order $r$ for $a$.

    An arbitrary sequence 
    \[ b(n) = c(jn^2+kn+\ell) \]
    can be written as interlacing of sequences 
    \[ a_r(n) = c(d(djn^2+(2jr+k)n)+jr^2+kr+\ell) \]
    for $r=0,\dots,d-1$ as the term at index $n=qd+r$ of the interlacing is precisely given by 
    \begin{align*}
        a_r(q) &= c(d(djq^2+(2jr+k)q)+jr^2+kr+\ell) \\
        &= c(j(d^2q^2+2rq+r^2)+k(dq+r)+\ell) = c(jn^2+kn+\ell).
    \end{align*}
    We can compute $C^2$-finite recurrences of order~$r$ for these sequences $a_r$ by the first part of the proof (choosing $j=dj,k=2jr+k,\ell=jr^2+kr+\ell$). By Theorem~\ref{theorem:interlacing} we can therefore compute a $C^2$-finite recurrence of order $dr$ for $b$.
\end{proof}

\begin{ex}
    Let $c$ be the $C$-finite sequence (\href{http://oeis.org/A006131}{A006131} in the OEIS) satisfying 
    \[ 4\, c(n)+c(n+1)-c(n+2)=0, \quad c(0)=c(1)=1. \]
    The sequence has eigenvalues~$\tfrac{1\pm \sqrt{17}}{2}$ and their torsion number is~$1$. The sparse subsequence~$a(n)=c(n^2)$ is $C^2$-finite of order~$2$ satisfying 
    \[ c_0(n) a(n) - c(4n+3) a(n+1) + c(2n) a(n+2)=0 \]
    where $c_0$ is $C$-finite of order~$2$ satisfying
    \[ 4096 \, c_0(n) -144 \, c_0(n+1) + c_0(n+2) = 0, \quad c_0(0)=-20, c_0(1)=-1856. \]
\end{ex}


\section{Outlook}
Recently, the class of \emph{simple $C^2$-finite sequences} has been
introduced~\cite{nuspl22b} that satisfies the same computational properties as
$C^2$-finite sequences, but does not share the same technical issues. In
particular, it is possible to derive bounds for the asymptotic behavior, there
is a characterization through the generating function and closure properties can
be computed more efficiently.

It is, however, not clear whether it is possible to derive order bounds for
simple $C^2$-finite sequences as we have presented here for $C^2$-finite
sequence. In that case, one is dealing with an inhomogeneous linear system and
the underlying ring is not a principal ideal domain. Hence, one cannot simply
bound the rank of modules.

Typically, given a defining recurrence for a $C^2$-finite sequence, it is
difficult to argue that it does not satisfy a shorter recurrence. For $D$-finite
sequences, it is a common strategy to use Guess-and-prove to derive a shorter
recurrence (or to find evidence that it is holonomic in the first place). It
would be desirable to have a guessing routine for $C^2$-finite sequences. As a
naive approach leads to a non-linear system (see also \cite{thanatipanonda20}),
it needs to be investigated how this can be solved efficiently.


\subsubsection*{Acknowledgment} 

We warmly thank Thierry Combot and Carsten Schneider for their advice.

\bibliographystyle{plain}
\bibliography{sources}

\end{document}